\documentclass[11pt]{amsart}

\usepackage[utf8]{inputenc}
\usepackage[english]{babel}
\usepackage{amsmath, amsthm, amssymb}
\usepackage{mathtools}
\usepackage{mathrsfs}
\usepackage[mathscr]{euscript}
\usepackage{stmaryrd}
\usepackage{enumerate}
\usepackage[all]{xy}
\usepackage{tikz-cd}
\usetikzlibrary{matrix,arrows}
\usepackage{rotating}
\usepackage{extarrows}
\usepackage{accents}
\usepackage{leftidx}
\usepackage{upgreek}
\usepackage{bbm}
\usepackage{bbold}
\usepackage{xr-hyper}
\usepackage{hyperref}
\usepackage{macros}

\hypersetup{  
 pdftitle    = {Conic bundles and iterated root stacks},
  pdfauthor   = {Jakob Oesinghaus},
  colorlinks=true}

\title{Conic bundles and iterated root stacks}

\begin{document}
\begin{abstract}
We generalize a classical result by V. G. Sarkisov about conic bundles to the case of a not necessarily algebraically closed perfect field, using iterated root stacks, destackification, and resolution of singularities. More precisely, we prove that whenever resolution of singularities is available, over a general perfect base field, any conic bundle is birational to a standard conic bundle.
\end{abstract}

\author{Jakob Oesinghaus}

\maketitle
\tableofcontents

\section{Introduction}\label{section_intro}
In this paper, we study the geometry of conic bundles, that is, fibrations whose generic fiber is a smooth conic. They have been widely studied in the context of rationality problems, notably the classic result of Artin and Mumford, who computed their Brauer groups in \cite{ArtinMumfordRational} to produce unirational non-rational conic bundles over rational surfaces. In order to better understand these bundles, it is desirable to bring a conic bundle into a standard form where the locus of degeneration can be controlled.

Over an algebraically closed field $k$ and assuming resolution of singularities, a classical result by Sarkisov (\cite{SarkisovConicB}) states every conic bundle $\pi:V \to S$  of irreducible varieties such that $S$ is smooth and $\pi$ is projective can be brought into a standard form. Concretely, this means that there exist smooth varieties $\widetilde{V}$ and $\widetilde{S}$ and a projective morphism $\widetilde{\pi}: \widetilde{V}\to \widetilde{S}$ fitting into a commutative square
\[
\begin{tikzcd} 
\widetilde{V} \arrow[d, "\widetilde{\pi}"] \arrow[r, dashed]   & V\arrow[d, "\pi"]\\
\widetilde{S} \arrow[r] & S
\end{tikzcd}
\]
such that the rational map $\widetilde{V} \dashrightarrow V$ and the projective morphism $\widetilde{S}\to S$ are birational, and such that
\begin{itemize}
\item the generic fiber of $\widetilde{\pi}$ is a smooth conic,
\item the discriminant divisor of $\widetilde{\pi}$ is a simple normal crossing divisor,
\item the general fiber of $\widetilde{\pi}$ along every irreducible component of the discriminant divisor is a singular irreducible reduced conic, and
\item the fibers of $\widetilde{\pi}$ over the singular locus of the discriminant divisor are non-reduced conics, i.e. double lines.
\end{itemize}

\noindent We use root stacks, resolution of singularities, and the destackification procedure (\cite{BerghDestackify}) to generalize Sarkisov's result to a general perfect base field in Theorem \ref{thm_main}. An analogous result for Brauer-Severi surface bundles, i.e. fibrations whose generic fiber is a form of $\bP^2$, has been proven by Kresch and Tschinkel in \cite{KrTschBrauerSeveri}, also using root stack techniques. 
\subsubsection*{Acknowledgements}
I would like to thank my advisor Andrew Kresch for his guidance. I am supported by Swiss National Science Foundation grant 156010.

\vspace{1em}

\section{Preliminaries}\label{section_prelim}
Throughout this section, when $X$ is a Noetherian Deligne-Mumford (DM) stack, we will let $n$ be a positive integer, invertible in the local rings of an \'etale atlas of $X$. A \emph{sheaf} on $X$ is a sheaf of abelian groups on the \'etale site of $X$; all cohomology will be \'etale cohomology.

\subsection{Conic bundles}
\begin{definition}
Let $S$ be a regular scheme such that $2$ is invertible in its local rings. A \emph{regular conic bundle over S} is a flat, projective morphism $\pi: V \to S$ from a regular scheme $V$ such that the generic fiber over every irreducible component is smooth and such that all fibers are isomorphic to a conic in $\bP^2$. A regular conic bundle is called \emph{standard} if $\pi$ is relatively minimal, i.e. if the preimage of an irreducible divisor under $\pi$ is an irreducible divisor; equivalently, if there exists a reduced divisor $D\subset S$ whose singular locus is regular, such that
\begin{itemize}
\item The morphism $\pi$ is smooth over $S \setminus D$ and the generic fiber over every irreducible component is a smooth conic.
\item The generic fiber of $\pi$ over every irreducible component of $D$ is a singular reduced irreducible conic, i.e. the union of two lines with conjugate slopes.
\item The fiber of $\pi$ over every point of $D^\sing$ is non-reduced, i.e. a double line.
\end{itemize}
\end{definition}

We remark that we put no further requirements on $D$, although our construction actually produces standard conic bundles with simple normal crossing discriminant divisor.

\subsection{Gerbes}
Let $X$ be a Noetherian Deligne-Mumford stack.
\begin{definition}
A gerbe over $X$ banded by $\mu_n$, or simply a $\mu_n$-gerbe over $X$, is the data of a Deligne-Mumford stack $H$ and a morphism $H \to X$ that is \'etale locally isomorphic to a product with $B\mu_n$, together with compatible identifications of the automorphism groups of local sections with $\mu_n$.
\end{definition}
We can classify $\mu_n$-gerbes by their class in $H^2(X,\mu_n)$. We use the notion of the \emph{residual gerbe} $\cG_x$ of $X$ at a point $x\in |X|$, an \'etale gerbe over the residue field $\kappa(x)$ satisfying certain universal properties (\cite[App. B]{Rydh-devissage}).

\begin{lemma}\label{lemma_coh_residual_gerbe}
Let $x\in |X|$. Then $H^1(\cG_x, \bZ) = 0$.
\end{lemma}
\begin{proof}The Leray spectral sequence for $f: \cG_x \to \Spec \kappa(x)$ gives a monomorphism
\[
H^1(\cG_x,\bZ) \to H^0(\Spec \kappa(x), R^1 f_*\bZ).
\]
Let $K/\kappa(x)$ be a finite separable extension such that the gerbe
\[
\cY := \Spec K\times_{\Spec \kappa(x)}\cG_x
\]
has a section. This implies that $\cY\cong\B G$ for a finite \'etale group scheme $G$ over $K$. But then
\[
H^1(\cY, \bZ) = H^1(\B G, \bZ) = \Hom(G,\bZ) = 0.
\]
Hence, we have $R^1f_*\bZ=0$.
\end{proof}

We will also frequently make use of the \emph{Kummer sequence}
\[
0 \to \mu_n \to \bG_m \xrightarrow{(\cdot)^n} \bG_m \to 0,
\]
which is an exact sequence of sheaves on $X$.

Suppose now that $X$ is regular and integral, with trivial stabilizer at the generic point $\iota_\eta :\Spec(\eta) \to X$. By \cite[(2)-(3)]{GB2}, the following is an exact sequence of sheaves on $X$:
\begin{equation}\label{eq_exact_sequence_iotaeta}
0 \to \bG_m \to (\iota_\eta)_*\bG_m \to \bigoplus_{x\in X^{(1)}} (\iota_x)_*\bZ \to 0.
\end{equation}
By Lemma \ref{lemma_coh_residual_gerbe}, this implies that $H^2(X,\bG_m) \to H^2(X,(\iota_\eta)_*\bG_m)$ is injective. The Leray spectral sequence for $\iota_\eta$ and Hilbert's Theorem 90 imply that
\[
H^2(X, (\iota_\eta)_*\bG_m) \to H^2(\Spec(\eta), \bG_m)
\]
is injective. Hence, we can infer that $H^2(X,\bG_m)$ is a torsion group, which we call the \emph{Brauer group} of $X$. This motivates the following definition.
\begin{definition}
The \emph{Brauer group} $\Br(X)$ of a Noetherian DM stack $X$ is defined to be the torsion subgroup of $H^2(X,\bG_m)$.
\end{definition}
This reduces to the classical definition of the Brauer group	 if $X$ is the spectrum of a field.

\subsection{Root stacks and the Brauer group}
Given effective Cartier divisors
\[
D_1, \dots, D_\ell
\]
on a Noetherian scheme or Deligne-Mumford stack $X$, we can define the \emph{iterated $n$-th root stack} of $X$ along those divisors (\cite{CadmanTangency}), denoted by
\[
\sqrt[n]{(X,\{D_1,\dots,D_\ell\})} \to X.
\]
This construction adds stacky structure along the divisors, and is an isomorphism outside of the union of divisors. It should be noted that when any intersection $D_i\cap D_j$ with $i\ne j$ is nonempty, this is not isomorphic to the root stack along the union of the divisors. For any $i\in \{1,\dots,\ell\}$, the iterated root stack 
\[
\sqrt[n]{(X,\{D_1,\dots,\widehat{D_i},\dots,D_\ell\})}
 \]
is a relative coarse moduli space for $\sqrt[n]{(X,\{D_1,\dots,D_\ell\})}$ in the sense of \cite[\S 3]{TwistedStable}. In particular, a locally free sheaf on $\sqrt[n]{(X,\{D_1,\dots,D_\ell\})}$ such that the associated linear $\mu_n$-representation is trivial at a general point of any irreducible component of $D_i$ descends to
\[
\sqrt[n]{(X,\{D_1,\dots,\widehat{D_i},\dots,D_\ell\})}.
\]
To prove this, we can apply a relative version of \cite[Thm. 10.3]{AlperGoodModuli} for good moduli space morphisms, which are relative versions of good moduli spaces. Note that
\begin{enumerate}
\item the condition of trivial stabilizer actions at closed points is replaced by trivial action of the \emph{relative} inertia stack at closed points;
\item due to the construction of root stacks, this action will be trivial if it is trivial at a general point of every irreducible component of $D_i$;
\item since the stacks involved are tame, the relative coarse moduli space mentioned above is a good moduli space morphism.
\end{enumerate}
We will use this several times throughout.

\begin{lemma}\label{lemma_Brauer_root_stack}
Let $S$ be a regular integral Noetherian scheme of dimension $1$, let $n$ be invertible in the local rings of $S$, and let $D_1, \dots, D_\ell$ be distinct closed points of $S$. Then we have 
\[
\Br(\sqrt[n]{(S,\{D_1,\dots,D_\ell\}})[n] \cong \Br(S\setminus \{D_1,\dots,D_\ell\})[n].
\]
\end{lemma}
\begin{proof}
Let $X$ be any regular integral Noetherian DM stack with trivial generic stabilizer such that $\dim(X)=1$, and such that $n$ is invertible in the local rings of an \'etale atlas of $X$. Then the results from \cite[2.]{GB3} and the Leray spectral sequence for $\iota_\eta$ imply that $H^2(X, (\iota_\eta)_*\bG_m)[n] \cong \Br(\eta)[n]$. Moreover, if $x\in X^{(1)}$, the vanishing of $H^1(\cG_x, \bZ)$ implies that
\[
H^1(\cG_x, \bZ / n\bZ) \cong H^2(\cG_x, \bZ)[n].
\]
The long exact sequence of cohomology of \eqref{eq_exact_sequence_iotaeta} then gives rise to an exact sequence
\begin{equation}\label{eq_exact_sequence_iotaeta_1d}
0 \to \Br(X)[n] \to \Br(\eta)[n] \to \bigoplus_{x\in X^{(1)}} H^1(\cG_x, \bZ / n\bZ).
\end{equation}
Taking $X=\sqrt[n]{(S,\{D_1,\dots,D_\ell\})}$, with codimension $1$ point $x_i$ over $D_i$ for all $i$, we compare the exact sequence \eqref{eq_exact_sequence_iotaeta_1d} with the analogous exact sequence for $S$ (loc. cit.) to obtain the vanishing of the right-hand map in \eqref{eq_exact_sequence_iotaeta_1d} after projection to the factor $x_i$ for any $i$ \cite[\S 3.2]{PeriodIndexArithmeticSurface}. Comparison with the exact sequence for $S\setminus\{D_1,\dots,D_\ell\}$ gives the result.
\end{proof}

\begin{lemma}\label{lemma_Brauer_codim2}
Let $k$ be a field and let $X$ be integral, smooth and of finite type over $k$ with trivial generic stabilizer. For any positive integer $n$ with $\chara(k) \nmid n$ and open substack $U \subset X$ whose complement has codimension at least $2$, we have $H^2(X,\mu_n)\cong H^2(U, \mu_n)$, and therefore $\Br(X)[n] \cong \Br(U)[n]$.
\end{lemma}
\begin{proof}
By \cite[Rem II.3.17]{MilneEtale} there is no loss of generality in assuming that $k$ is perfect. By shrinking $X$ if necessary and iterating the process for large open substacks of $X$, we can assume that the complement $Y=X \setminus U$ is smooth and of constant codimension $c\geq 2$ everywhere. In this situation, we know by \cite[\S XVI.3]{SGA4-3} that $\underline{H}^i_Y(X,\mu_n) = 0$ for $i\neq 2c$. Combining this with the exact sequence for cohomology with support (\cite[Prop III.1.25]{MilneEtale}) and the local-to-global spectral sequence (\cite[\S VI.5]{MilneEtale}) gives the result; cf. \cite[Cor 6.2]{GB3}.
\end{proof}

\section{Proof the of the main result}\label{section_proof}
Here we state and prove the main result.

\begin{theorem}\label{thm_main}
Let $k$ be a perfect field of characteristic different from $2$ and $S$ a smooth projective algebraic variety over $k$.
Assume that embedded resolution of singularities for reduced subschemes of $S$ of pure codimension $1$ and desingularization of reduced finite-type Deligne-Mumford stacks of pure dimension equal to $\dim(S)$ are known.
Let
\[
\pi:V\to S
\]
be a morphism of projective varieties over $k$ whose generic fiber is a
smooth conic.
Then there exists a commutative diagram
\[
\begin{tikzcd} 
\widetilde{V} \arrow[d, "\widetilde{\pi}"] \arrow[r, dashed, "\rho_V"]   & V\arrow[d, "\pi"]\\
\widetilde{S} \arrow[r, "\rho_S"] & S
\end{tikzcd}
\]
where $\rho_S$ is a projective birational morphism,
$\rho_V$ is a birational map, and
$\widetilde\pi$ is a standard conic bundle with simple normal crossing discriminant divisor.
\end{theorem}
\begin{remark}
Embedded resolution of singularities is known in characteristic $0$ for all dimensions by Hironaka's celebrated result. In positive characteristic, embedded resolution of singularities for both curves and surfaces is known (cf. \cite{CJS-ResSing}). Since the resolutions commute with smooth base change, the assumptions about desingularization of reduced finite-type Deligne-Mumford stacks are also true in all of these cases (apply resolution of singularities to a presentation).
\end{remark}

Before embarking on the proof we make several observations of a general nature.
Let $k$ be a perfect field, let $S$ be a smooth projective algebraic variety over $k$, let $n$ be a positive integer such that $\chara(k) \nmid n$, and let $\alpha\in\Br(k(S))[n]$. Then there exists a dense open $U\subset S$ and $\beta \in \Br(U)$ such that $\alpha $ is the restriction of $\beta$. Taking $U$ to be maximal, by Lemma \ref{lemma_Brauer_codim2}, the complement $S \setminus U$ is a finite union of divisors.
\begin{assumption}
Assume embedded resolution of singularities for reduced subschemes of $S$ of pure codimension $1$.
\end{assumption}

Then, upon replacing $S$ by a smooth projective variety with birational morphism to $S$, we may suppose that the complement of $U$ is a simple normal crossing divisor $D_1 \cup \dots \cup D_\ell$. Let
\[
X := \sqrt[n]{(S,\{D_1,\dots,D_\ell\})},
\]
the iterated root stack of $S$ along the divisors $D_i$. We apply Lemma \ref{lemma_Brauer_root_stack} to the scheme obtained by gluing the local rings at the generic points of the $D_i$ along the generic point of $S$. Then \cite[\S VII.5, Thm 5.7]{SGA4-2} implies that $\alpha$ extends to an open neighborhood of the root stack over this scheme. Hence, by Lemma \ref{lemma_Brauer_codim2}, there is a unique $\beta\in\Br(X)[n]$ that restricts to $\alpha$.

Now suppose that $\alpha$ is the class of a central simple algebra $A$ of dimension $d^2$ as a $k(S)$-vector space, with $n\mid d$. Let $\beta_0\in H^2(X, \mu_n)$ be a lift of $\beta\in\Br(X)[n]$ with corresponding gerbe $G_0$ banded by $\mu_n$. Now $A$ is the fiber at the generic point of some sheaf of Azumaya algebras $\cA$ on an open $W\subset X$ with complement of codimension at least $3$ (\cite[Thm 2.1]{GB2}). The Brauer class of the pullback of $\cA$ to $W\times_X G_0$ is trivial, hence this pullback is the endomorphism algebra of a locally free coherent sheaf of rank $d$, which is the restriction of a coherent sheaf $\cE_0$ on $G_0$.
\begin{assumption}
Assume resolution of singularities for reduced Noetherian DM stacks of finite type over $k$ of pure dimension $\dim(S)$.
\end{assumption}
The identity of $\cE_0\vert{W\times_X G_0}$ induces a morphism to the Grassmannian of rank $d$ quotients of $\cE_0$ (\cite{NitsureQuot}). Apply resolution of singularities to the closure of the image to obtain a smooth DM stack $Y$ with a projective morphism to $X$ that restricts to an isomorphism over $W$, a gerbe $H_0:=Y\times_X G_0$, and a locally free coherent sheaf $\cF_0$ on $H_0$ whose restriction over $W$ is isomorphic to the restriction of $\cE_0$. In this situation, let $\gamma_0 := \beta_0\vert_Y$.

In the proof of Theorem \ref{thm_main} we specialize the above to $n=d=2$.

\begin{proof}
The proof begins with a series of reductions steps, starting with $Y$ as above, equipped with a sheaf of Azumaya algebras, restricting to the quaternion algebra $A$ over $k(S)$, associated with the generic fiber of $\pi$.

\emph{Step 1.} We may suppose that $Y\cong \sqrt{(T,\{E_1,\dots,E_m\})}$ for some smooth projective variety $T$ with birational morphism to $S$ and irreducible divisors $E_i$ such that $E_1\cup\dots\cup E_m$ is a simple normal crossing divisor and such that at the generic point of each $E_i$, the projective representation $\mu_2\to \PGL_2$ given by the sheaf of Azumaya algebras is nontrivial.
Indeed, the destackification program (\cite[Thm 1.2]{BerghDestackify}) yields a morphism $\widetilde{Y}\to Y$ that is a composition of blow-ups with smooth centers, such that
\[
\widetilde{Y} \cong \sqrt{(T,\{E_1,\dots,E_m\})}
\]
for a smooth projective variety $T$ and irreducible divisors $E_i$ such that $E_1\cup\dots\cup E_m$ is a simple normal crossing divisor.
We pull back the sheaf of Azumaya algebras to $\widetilde{Y}$. If there is an $i$ such that the projective representation $\mu_2\to \PGL_2$ over a general point of $E_i$ is trivial, the sheaf of Azumaya algebras descends to
\[
\sqrt{(T,\{E_1,\dots,\widehat{E_i},\dots,E_m\})}.
\]

\emph{Step 2.} We may suppose, additionally, that generically along every component of $E_i\cap E_{i'}$ for $i\ne i'$ the projective representation of $\mu_2\times \mu_2$ is faithful. Let $F\subset E_i\cap E_{i'}$ be an irreducible component with non-faithful representation. Let $\widetilde{T}$ be the blow-up of $T$ along $F$. For every $j\in\{1,\dots,m\}$, we denote the proper transform of $E_j$ by $\widetilde{E}_j$, and we denote the exceptional divisor of the blow-up by $E'$. We let $\widetilde{Y}$ be the normalization of $\widetilde{T} \times_T Y$.
Then $\widetilde{Y}$ is isomorphic to the blow-up of $Y$ at the corresponding component of the fiber product of the gerbes of the root stacks, which is itself isomorphic to the root stack
\[
\sqrt{(\widetilde{T},\{\widetilde{E}_1,\dots,\widetilde{E}_m,E'\})}.
\]
The projective representation over a general point of $E'$ is trivial, so the sheaf of Azumaya algebras descends to
\[
\sqrt{(\widetilde{T},\{\widetilde{E}_1,\dots,\widetilde{E}_m\})}.
\]

\emph{Step 3.} We may suppose, furthermore, that all triple intersections $E_i\cap E_{i'}\cap E_{i''}$ are empty, where $i$, $i'$, and $i''$ are distinct. Since there can never be more than $2$ independent commuting subgroups of order $2$ in $PGL_2$ (\cite{BeauvilleFinite}), the projective representation $(\mu_2)^3\to PGL_2$ over a general point of $E_i \cap E_{i'} \cap E_{i''}$ has kernel equal to the diagonal $\mu_2$.
We blow up $T$ along $E_i \cap E_{i'} \cap E_{i''}$ and proceed as in Step 2.

\emph{Step 4.} We may suppose, furthermore, that the Brauer class $[A]\in \Br(k(S))$ does not extend across the generic point of $E_i$ for any $i$. Assume that it does, for some $i$. Let 
\[
\widetilde{Y} := \sqrt{(T,\{E_1,\dots,\widehat{E_i},\dots,E_m\})}.
\]
Then by Lemma \ref{lemma_Brauer_codim2}, the Brauer class is restriction of an element $\delta \in \Br(\widetilde{Y})$. Let $\varepsilon\in H^2(\widetilde{Y}, \mu_2)$ denote an arbitrary lift of $\delta$, with corresponding gerbe $H_1$, such that if we let $H_0$ denote the base-change
\[
Y\times_{\widetilde{Y}} H_1,
\]
then on $H_0$ the sheaf of Azumaya algebras is identified with endomorphism algebra of some locally free coherent sheaf $\cF_0$.
Notice that $H_0$ is a root stack over $H_1$.
The relative stabilizer acts with eigenvalues $1$ and $-1$ on fibers of $\cF_0$.
The $(-1)$-eigensheaf is a quotient sheaf $\cL_{-1}$ supported on the gerbe of the root stack, such that the kernel $\cF_1$ in
\[
0\to \cF_1\to \cF_0\to \cL_{-1}\to 0
\]
is again locally free and descends to $H_1$, yielding a sheaf of Azumaya algebras on $\widetilde{Y}$.
\vspace{1em}

We now have $Y = \sqrt{(T,\{E_1,\dots,E_m\})}$, equipped with a sheaf of Azumaya algebras $\cA$, such that the projective representations at a general point of every $E_i$ are nontrivial, the projective representations at a general point of every intersection is faithful, there are no triple intersections, and such that the Brauer class does not extend over any of the generic points of the $E_i$. Let $P\to Y$ be the smooth $\bP^1$-fibration associated with $\cA$.

Let $T_0$ denote the complement of the intersections of pairs of divisors,
\[
T_0:=T\setminus \bigcup_{1\le i<i'\le m} E_i\cap E_{i'}.
\]
We apply \cite[Proposition 3.1]{KrTschBrauerSeveri} to $T_0\times_TP$ to obtain a regular conic bundle
\[
\pi_0\colon V_0\to T_0.
\]
This factors canonically through $\bP(\pi_{0{*}}(\omega_{V_0/T_0}^\vee))$.
Let $i:T_0\to T$ denote the inclusion. We claim that $i_*(\pi_{0{*}}(\omega_{V_0/T_0}^\vee))$
is a locally free coherent sheaf and, denoting this by $\cE$, the closure
$V=\overline{V_0}$ of $V_0$ in $\bP(\cE)$ is a regular conic bundle over $T$.
It suffices to verify these assertions after passing to an algebraic closure of $k$.
Then there is a unique faithful projective representation
$\mu_2 \times \mu_2\to \PGL_2$ (up to conjugacy), cf. \cite{BeauvilleFinite}.
So, by \cite[Lemma 2.8]{KrTschBrauerSeveri} after base change to a suitable
affine \'etale neighborhood $T'=\Spec(B')\to T$ of a given point of an intersection
$E_i\cap E_{i'}$, we have
\[
Y'\cong \sqrt{(T',\{E_i,E_{i'}\})}=[\Spec(B'[t,t']/(t^2-x,t'^2-x'))/\mu_2\times \mu_2],
\]
where $x$ and $x'$ are the respective defining equations for the preimage in
$T'$ of $E_i$ and $E_{i'}$, with $P'$ obtained by pulling back
$[\bP^1/\mu_2\times \mu_2]$. Here, on $B'[t,t']/(t^2-x,t'^2-x')$, the action
of the factors of $\mu_2\times \mu_2$ is by respective scalar multiplication of
$t$ and $t'$, while the action on $\bP^1$ corresponds to the faithful projective
representation $\mu_2 \times \mu_2\to \PGL_2$. Over $T'$ we compute
$\overline{V'_0}\cong\Proj(B'[u,v,w]/(xu^2+x'v^2-w^2))$, which is regular.

The fact that the Brauer class does not extend across the generic point of any $E_i$ ensures that the conic bundle $V\to T$ is standard.
\end{proof}
\begin{remark}
While the general destackification process outlined in \cite{BerghDestackify} requires stacky blow-ups, it is never necessary to take root stacks when all stabilizers are powers of $\mu_2$.
\end{remark}

\bibliographystyle{hplain}
\bibliography{bib}

\end{document}